\documentclass[a4paper,10pt]{article}
\usepackage{amsmath,amssymb,amsthm}
\usepackage{amsfonts}
\newtheorem{theorem}{Theorem}[section]

\newtheorem{corollary}[theorem]{Corollary}

\newtheorem{lemma}[theorem]{Lemma}

\newcommand*\diff{\mathop{}\!\mathrm{d}}
\numberwithin{equation}{section} \makeatletter
\date{} 
\title{On Visser's inequality concerning coefficient estimates for a polynomial}
\author{ S. Gulzar$^\dagger,$ N.A. Rather$^\ddagger$ \& M.S. Wani$^\ddagger$} 
\begin{document}
\maketitle
\begin{center} 
$^\dagger$Department of Mathematics,\\ Government College for Engineering and Technology, J\& K, India\\
$^\ddagger$Department of Mathematics, \\University of Kashmir, Srinagar-190006, India
\end{center} 
\begin{center}
e-mail:   sgmattoo@gmail.com
\end{center}
\begin{abstract}
If $P(z)=\sum_{j=0}^{n}a_jz^j$ is a polynomial of degree $n$ having no zero in $|z|<1,$ then it was recently proved that for every $p\in[0,+\infty]$ and $s=0,1,\ldots,n-1,$
\begin{align*}
\left\|a_nz+\frac{a_s}{\binom{n}{s}}\right\|_{p}\leq \frac{\left\|z+\delta_{0s}\right\|_p}{\left\|1+z\right\|_p}\left\|P\right\|_{p},
\end{align*}
where $\delta_{0s}$ is the Kronecker delta. In this paper, we consider the class of polynomials having no zero in $|z|<\rho,$ $\rho\geq 1$ and obtain some generalizations of above inequality.
\end{abstract}
\section{Introduction}
Let $\mathcal{P}_n$ be the set of all polynomials of degree $n.$ For $P\in\mathcal{P}_n$ define
\begin{align*}
\|P\|_p&:=\left(\frac{1}{2\pi}\int_{0}^{2\pi}|P(e^{i\theta})|^p\diff\theta\right)^{1/p}\qquad (0<p<\infty),\\
&\|P\|_\infty:=\underset{|z|=1}{\max}|P(z)|\qquad\text{and}\\
\|P\|_0&:=\exp\left(\frac{1}{2\pi}\int_{0}^{2\pi}\log|P(e^{i\theta})|\diff\theta\right).
\end{align*}
If $P\in\mathcal{P}_n$ and $P(z)=\sum_{j=0}^{n}a_jz^j,$ then
\begin{align}\label{e1}
|a_n|+|a_0|\leq \|P\|_\infty.
\end{align} 
This inequality is due to Visser's inequality \cite{v} and it is well known that equality holds in \eqref{e1} only when $a_j=0$ for $j=1,2,\ldots,n-1.$ We refer the reader to the book of Rahman and Schmeisser \cite{rs} for a survey of extensions of Visser's inequality.

Recently, following extenstion of \eqref{e1} in integral norm was proved in \cite{sg}.
\begin{theorem}\label{t1}
If $P\in\mathcal{P}_n$ and $P(z)=\sum_{j=0}^{n}a_jz^j,$ then for $0\leq p< \infty,$
\begin{align}\label{t1e}
\left\|a_nz+\frac{a_s}{\binom{n}{s}}\right\|_{p}\leq \left\|P\right\|_p, \qquad s=0,1,\ldots,n-1.
\end{align}
\end{theorem}
The case $s=n-1$ was independently obtained in \cite{sam}.

For the class of polynomials having no zero in $|z|<1,$ in \cite{sg} the following improvement of Theorem \ref{t1} was also obtained.
\begin{theorem}\label{t2}
If $P\in\mathcal{P}_n$ and $P(z)=\sum_{j=0}^{n}a_jz^j$ does not vanish in $|z|<1,$ then for $0\leq p<\infty,$
\begin{align*}
\left\|a_nz+\frac{a_s}{\binom{n}{s}}\right\|_{p}\leq c_p\left\|P\right\|_p, \qquad s=0,1,\ldots,n-1.
\end{align*}
\begin{equation}\label{cp}
c_p=
\begin{cases}
\quad\quad1\quad\quad\text{if}\quad s=0\\
~~\dfrac{1}{\|1+z\|_p}\quad ~\text{if}\quad s>0.
\end{cases}
\end{equation}
\end{theorem}
It is worth to mention here that the above result is also implicit in the main result of \cite{a2}. In this paper, Arestov studied inequalities for Schur-Szeg\"{o} composition of $\Lambda_\gamma(z)=\sum_{j=0}^n\binom{n}{j}\gamma_jz^j$ having all zeros in $|z|\leq 1$ and a polynomial $P(z)$ with no zero in $|z|<1.$ It was proved that for any function $\varphi(x)=\psi(\log x),$ where $\psi$ is convex and non-decreasing function on $\mathbb{R},$ 
\begin{align}\nonumber\label{eqtn1}
\int_{0}^{2\pi}\int_{0}^{2\pi}\varphi &\left(|(1+e^{i\theta})\Lambda_\gamma P(e^{it})|\right)dtd\theta\\ &\leq \int_{0}^{2\pi}\int_{0}^{2\pi}\varphi\left(|(\gamma_0+e^{i\theta}\gamma_n)P(e^{it})|\right)dtd\theta.
\end{align}
Note that Theorem \ref{t2} follows from \eqref{eqtn1} by taking $\Lambda_\gamma=z^n+z^s$ and $\varphi(x)=x^p,$ when $p>0$ and $\varphi(x)=\log x$ if $p=0.$
\section{Main results}
In this paper, we consider the class of polynomials having no zero in $|z|<\rho$ where $\rho\geq 1$ and obtain a generalization of Theorem \ref{t2}. In this direction, we prove:
\begin{theorem}\label{t3}
If $P\in\mathcal{P}_n$ and $P(z)=\sum_{j=0}^{n}a_jz^j$ does not vanish in $|z|<\rho$ where $\rho\geq 1$ then for $0\leq p<\infty,$
\begin{align}\label{t3e}
\left\|a_nz+\frac{a_s}{\binom{n}{s}}\right\|_p\leq c_p\frac{\left\|1+z\right\|_p}{\left\|\rho^s+z\right\|_p}\left\|P\right\|_p,\qquad s=0,1,\ldots,n-1
\end{align}
where $c_p$ is given by \eqref{cp}.  Moreover, if $a_j=0$ for $j=1,2,\ldots,n-1$ and $s=0,$ then equality holds in \eqref{t3e}.
\end{theorem}
The first Corollary is obtained by letting $p\rightarrow \infty$ in \eqref{t3e}.
\begin{corollary}
If $P\in\mathcal{P}_n$ and $P(z)=\sum_{j=0}^{n}a_jz^j$ does not vanish in $|z|<\rho$ where $\rho\geq 1$ then for $s=1,\ldots,n-1$
\begin{align}
|a_n|+\frac{|a_s|}{\binom{n}{s}}\leq \frac{\left\|P\right\|_\infty}{1+\rho^s}. 
\end{align}
\end{corollary}
By using inequality \eqref{t3e} in conjunction with Lemma \ref{sl1}, we obtain following generalization of Theorem \ref{t2}.
\begin{corollary}\label{c2}
If $P\in\mathcal{P}_n$ and $P(z)=\sum_{j=0}^{n}a_jz^j$ does not vanish in $|z|<\rho$ where $\rho\geq 1$ then for $0\leq p<\infty,$
\begin{align}\label{c2e}
|a_n|+\frac{|a_s|}{\binom{n}{s}}\leq 2c_p\frac{\left\|P\right\|_p}{\left\|\rho^s+z\right\|_p}
\end{align}
where $s=0,1,\ldots,n-1$ and 
$c_p$ is given by \eqref{cp}.
\end{corollary}

\section{Lemmas}
\begin{lemma}\label{nl1}
If $P\in\mathcal{P}_n$ and $P(z)=\sum_{j=0}^{n}a_jz^j$ does not vanish in $|z|<\rho$ where $\rho \geq 1$ then 
\begin{align}
\rho^s\left|a_nz^n+\frac{a_s}{\binom{n}{s}}z^s\right|\leq\left|\frac{a_s}{\binom{n}{s}}z^{n-s}+a_0\right|\qquad\text{for}\quad |z|=1.
\end{align}
\end{lemma}
\begin{proof}
Let $F(z)=P(\rho z).$ Since $P(z)$ has all its zeros in $|z|\geq \rho,$ then all the zeros of $F^{\star}(z)=z^n\overline{F(1/\overline{z})}$ lie in $|z|\leq 1$ and $|F^{\star}(z)|=|F(z)|$ for $|z|=1.$ Since $F(z)$ does not vanish in $|z|<1,$ therefore $F^{\star}(z)/F(z) $ is analytic in $|z|\leq 1.$ By maximum modulus principle $|F^{\star}(z)|\leq |F(z)|$ for $|z|\leq 1$ or equivalently $|F(z)|\leq |F^{\star}(z)|$ for $|z|\geq 1.$ By using Rouch\'{e}'s theorem it follows that all the zeros of polynomial $$F(z)+\lambda F^{\star}(z)=\sum_{j=0}^{n}(\rho^ja_{j}+\lambda \rho^{n-j}\overline{a}_{n-j})z^{j}$$ lie in $|z|\leq 1$ for every $\lambda\in\mathbb{C}$ with $|\lambda|> 1.$ If $z_1,z_2,\ldots,z_n$ are zeros of $F(z)+\lambda F^{\star}(z),$ then $|z_j|\leq 1,$ $j=1,2,\ldots,n$ and we have by Vi\`{e}te's formula for $s=0,1,\ldots,n-1,$
\begin{align*}
(-1)^{n-s}\left(\frac{\rho^{s}a_s+\lambda \rho^{n-s}\bar{a}_{n-s}}{\rho^{n}a_n+\lambda \bar{a}_0}\right)=\sum\limits_{1\leq i_{1}< i_{2}\ldots<i_{n-s}\leq n}z_{i_{1}}z_{i_{2}}\ldots z_{i_{n-s}}
\end{align*}
which gives
\begin{align}\label{l1eq1ee}\nonumber
\left|\frac{\rho^{s}a_s+\lambda \rho^{n-s}\bar{a}_{n-s}}{\rho^na_n+\lambda \bar{a}_0}\right|&\leq\sum\limits_{1\leq i_{1}< i_{2}\ldots<i_{n-s}\leq n}\left|z_{i_{1}}z_{i_{2}}\ldots z_{i_{n-s}}\right|\\ &\leq \binom{n}{n-s}=\binom{n}{s}.
\end{align}
Therefore, all the zeros of polynomial 
\begin{align*}
M(z)&=(\rho^na_n+\lambda \bar{a}_0)z^n+\frac{\rho^sa_s+\lambda \rho^{n-s} \bar{a}_{n-s}}{\binom{n}{s}}z^s\\&=\rho^na_nz^n+\rho^s\frac{a_s}{\binom{n}{s}}z^s+\lambda\left(\bar{a}_0z^n+\rho^{n-s}\frac{\bar{a}_{n-s}}{\binom{n}{s}}z^s\right)
\end{align*}
lie in $|z|\leq 1$ for $\lambda\in\mathbb{C}$ with $|\lambda|>1$ and hence for  $r>1,$ the polynomial
\begin{align*}
M(rz)=\rho^na_n(rz)^n+\rho^s\frac{a_s}{\binom{n}{s}}(rz)^s+\lambda\left(\bar{a}_0(rz)^n+\rho^{n-s}\frac{\bar{a}_{n-s}}{\binom{n}{s}}(rz)^s\right)
\end{align*}
has all its zeros in $|z|<1.$ This implies 
\begin{align}\label{l1eq2ee}
\left|\rho^na_n(rz)^n+\rho^s\frac{a_s}{\binom{n}{s}}(rz)^s\right|\leq \left|\bar{a}_0(rz)^n+\rho^{n-s}\frac{\bar{a}_{n-s}}{\binom{n}{s}}(rz)^k\right|
\end{align}
for $|z|\geq 1.$ Indeed, if inequality \eqref{l1eq2ee} is not true, then there exists a point $w$ with $|w|\geq 1$ such that 
\begin{align*}
\left|\rho^na_n(rw)^n+\rho^s\frac{a_s}{\binom{n}{s}}(rw)^s\right|> \left|\bar{a}_0(rw)^n+\rho^{n-s}\frac{\bar{a}_{n-s}}{\binom{n}{s}}(rw)^k\right|
\end{align*}
Since all the zeros of $F^{\star}(z)=z^{n}\overline{P(\rho/\overline{z})}=\bar{a}_0z^n+\bar{a}_1\rho z^{(n-1)}+\cdots+\bar{a}_{(n-s)}\rho^{(n-s)}z^{s}+\cdots+\bar{a}_n\rho^n $ lie in $|z|\leq 1,$ Vi\`{e}te's formula again yields, as before in  \eqref{l1eq1ee},  $|\bar{a}_0|\geq \rho^{n-s}|\bar{a}_{n-s}|/\binom{n}{s},$ which implies that $\bar{a_0}(rw)^{n}+\rho^{n-s}\bar{a}_{n-s}/\binom{n}{s}(rw)^{s}\neq 0.$ We take
$$\lambda=-\frac{\rho^na_n(rw)^{n}+\rho^s\frac{a_s}{\binom{n}{s}}(rw)^{s}}{\bar{a}_0(rw)^{n}+\rho^{n-s}\frac{\bar{a}_{n-s}}{\binom{n}{s}}(rw)^{s}}, $$
then $\lambda$ is a well defined real or complex number with $|\lambda|>1$ and with this choice of $\lambda,$ we obtain $M(rw)=0$ where $|w|\geq 1.$ This contradicts the fact that all the zeros of $M(rz)$ lie in $|z|<1.$ Thus \eqref{l1eq2ee} holds. Letting $r\rightarrow 1$ in \eqref{l1eq2ee} and using continuity, we obtain
 \begin{align}\label{l1eq2'}
\left|\rho^na_nz^{n-s}+\rho^s\frac{a_s}{\binom{n}{s}}\right|\leq \left|\bar{a}_0z^n+\rho^{n-s}\frac{\bar{a}_{n-s}}{\binom{n}{s}}z^s\right|=\left|\frac{\rho^{n-s}a_{n-s}}{\binom{n}{s}}z^{n-s}+a_0\right|
\end{align}
for $|z|= 1.$ Again, since $|\bar{a}_0|\geq \rho^{n-s}|\bar{a}_{n-s}|/\binom{n}{s}$ the polynomial $\frac{\rho^{n-s}a_{n-s}}{\binom{n}{s}}z^{n-s}+a_0$ does not vanish in $|z|<1.$ Hence by the maximum modulus principle, 
 \begin{align}\label{1z}
 \left|\rho^na_nz^{n-s}+\rho^s\frac{a_s}{\binom{n}{s}}\right|\leq \left|\rho^{n-s}\frac{a_{n-s}}{\binom{n}{s}}z^{n-s}+a_0\right|\quad \text{for}\quad |z|\leq 1.
 \end{align}
Taking $z=e^{i\theta}/\rho$ where $0\leq \theta<2\pi$ in \eqref{1z}, we get
\begin{align*}
\rho^s\left|a_ne^{i(n-s)\theta}+\frac{a_s}{\binom{n}{s}}\right|\leq \left|\frac{a_{n-s}}{\binom{n}{s}}e^{i(n-s)\theta}+a_0\right|.
\end{align*}
Equivalently,
\begin{align*}
\rho^s\left|a_nz^n+\frac{a_s}{\binom{n}{s}}z^s\right|\leq \left|\frac{a_{n-s}}{\binom{n}{s}}z^{n-s}+a_0\right|\qquad\text{for}~~|z|=1.
\end{align*}
This completes the proof of Lemma \ref{nl1}.
\end{proof}
We now describe a result of Arestov \cite{2}.

For $\boldsymbol\gamma = \left(\gamma_{0},\gamma_{1},\ldots,\gamma_{n}\right)\in\mathbb {C}^{n+1}\,\,\,\,\mbox{}\,\,\,\textrm{and}\,\,\,\,\,\,P(z)=\sum_{j=0}^{n}a_{j}z^{j} \in \mathcal P_{n}$, we define\\
$$C_{\boldsymbol\gamma}P(z)=\sum_{j=0}^{n}\gamma_{j} a_{j}z^{j}.$$
\indent The operator $C_{\boldsymbol\gamma}$ is said to be \textit{admissible} if it preserves one of the following properties:
\begin{enumerate}
\item[(i)] $P(z)$ has all its zeros in $|z|\leq 1,$
\item[(ii)] $P(z)$ has all its zeros in $|z|\geq 1.$ 
\end{enumerate}
The result of Arestov \cite[Theorem 2]{2} may now be stated as follows.
\begin{lemma}\label{l3}
 Let $\varphi(x)=\psi(\log x)$ where $\psi$ is a convex non-decreasing function on $\mathbb{R}$. Then for all $P\in \mathcal P_{n}$ and each admissible operator $C_{\boldsymbol\gamma}$,
\begin{equation}
\int_{0}^{2\pi}\varphi\left(|C_{\boldsymbol\gamma}P(e^{i\theta})|\right)d\theta \leq \int_{0}^{2\pi}\varphi\left(c(\boldsymbol\gamma)|P(e^{i\theta})|\right)d\theta\nonumber
\end{equation}
where $c(\boldsymbol{\gamma})= \max \left(|\gamma_{0}|,|\gamma_{n}|\right)$.
\end{lemma}
In particular Lemma \ref{l3} applies with $\varphi : x\mapsto x^{p}$ for every $ p \in(0,\infty)$ and with $\varphi : x\mapsto \log x$ as well. Therefore, we have for $0 \leq p <\infty$,
\begin{equation}\label{l3e}
 \left\{\int_{0}^{2\pi}\left|C_{\boldsymbol\gamma}P(e^{i\theta})\right|^{p}d\theta\right\}^{1/p}\leq c(\boldsymbol{\gamma})\left\{\int_{0}^{2\pi}\left|P(e^{i\theta})\right|^{p}d\theta\right\}^{1/p}.
\end{equation}

\begin{lemma}\label{l2}
If $P\in\mathcal{P}_n$ and $P(z)=\sum_{j=0}^{n}a_jz^j$ does not vanish for $|z|<1,$ then for  $k=0,1,\ldots,n-1,$ $\phi$ real and each $p>0,$
\begin{align*}\nonumber
\int_{0}^{2\pi}&\left|\left(a_ne^{in\theta}+\frac{a_k}{\binom{n}{k}}e^{ik\theta}\right)e^{i\phi}+\left(\frac{a_{n-k}}{\binom{n}{k}}e^{i(n-k)\theta}+a_0\right)\right|^{p}d\theta\\&\qquad\qquad\qquad\qquad\qquad\leq\Lambda^p\int_{0}^{2\pi}|P(e^{i\theta})|^pd\theta
\end{align*}
where $k=0,1,\ldots,n-1$ and
$\Lambda=
\begin{cases}
|1+e^{i\phi}|\qquad if\quad k=0\\
\qquad 1\quad \qquad if \quad 0<k<n
\end{cases}
.$
\end{lemma}
\begin{proof}
By hypothesis $P(z)$ has all its zeros in $|z|\geq 1,$ therefore all the zeros of $P^{\star}(z)=z^n\overline{P(1/\overline{z})}$ lie in $|z|\leq 1$ and $|P(z)|=|P^{\star}(z)|$ for $|z|=1.$ Therefore $P^{\star}(z)/P(z) $ is analytic in $|z|\leq 1.$ By maximum modulus principle $|P^{\star}(z)|\leq |P(z)|$ for $|z|\leq 1$ or equivalently $|P(z)|\leq |P^{\star}(z)|$ for $|z|\geq 1.$ An application of Rouch\'{e}'s theorem shows that all the zeros of polynomial $$P(z)+\lambda P^{\star}(z)=\sum_{j=0}^{n}(a_{j}+\lambda \overline{a}_{n-j})z^{j}$$ lie in $|z|\leq 1$ for every $\lambda\in\mathbb{C}$ with $|\lambda|> 1.$ If $z_1,z_2,\ldots,z_n$ are roots of $P(z)+\lambda P^{\star}(z),$ then $|z_j|\leq 1,$ $j=1,2,\ldots,n$ and we have by Vi\`{e}te's formula for $k=0,1,\ldots,n-1,$
\begin{align*}
(-1)^{n-k}\left(\frac{a_k+\lambda \bar{a}_{n-k}}{a_n+\lambda \bar{a}_0}\right)=\sum\limits_{1\leq i_{1}< i_{2}\ldots<i_{n-k}\leq n}z_{i_{1}}z_{i_{2}}\ldots z_{i_{n-k}},
\end{align*}
which gives
\begin{align}\label{l1eq1}
\left|\frac{a_k+\lambda \bar{a}_{n-k}}{a_n+\lambda \bar{a}_0}\right|\leq\sum\limits_{1\leq i_{1}< i_{2}\ldots<i_{n-k}\leq n}\left|z_{i_{1}}z_{i_{2}}\ldots z_{i_{n-k}}\right|\leq \binom{n}{n-k}=\binom{n}{k}.
\end{align}
Therefore, all the zeros of polynomial 
\begin{align*}
M(z)=(a_n+\lambda \bar{a}_0)z^n+\frac{a_k+\lambda \bar{a}_{n-k}}{\binom{n}{k}}z^k=a_nz^n+\frac{a_k}{\binom{n}{k}}z^k+\lambda\left(\bar{a}_0z^n+\frac{\bar{a}_{n-k}}{\binom{n}{k}}z^k\right)
\end{align*}
lie in $|z|\leq 1$ for $\lambda\in\mathbb{C}$ with $|\lambda|>1$ and hence for $r>1,$ the polynomial
\begin{align*}
M(rz)=a_n(rz)^n+\frac{a_k}{\binom{n}{k}}(rz)^k+\lambda\left(\bar{a}_0(rz)^n+\frac{\bar{a}_{n-k}}{\binom{n}{k}}(rz)^k\right)
\end{align*}
has all its zeros in $|z|<1.$ This implies 
\begin{align}\label{l1eq2}
\left|a_n(rz)^n+\frac{a_k}{\binom{n}{k}}(rz)^k\right|\leq \left|\bar{a}_0(rz)^n+\frac{\bar{a}_{n-k}}{\binom{n}{k}}(rz)^k\right|
\end{align}
for $|z|\geq 1.$ Indeed, if inequality \eqref{l1eq2} is not true, then there exists a point $w$ with $|w|\geq 1$ such that 
\begin{align*}
\left|a_n(rw)^{n}+\frac{a_k}{\binom{n}{k}}(rw)^{k}\right|> \left|\bar{a}_0(rw)^{n}+\frac{\bar{a}_{n-k}}{\binom{n}{k}}(rw)^{k}\right|.
\end{align*}
Since all the zeros of $P^{\star}(z)$ lie in $|z|\leq 1$ then by the similar argument as in  \eqref{l1eq1}, we have $|\bar{a}_0|\geq |\bar{a}_{n-k}|/\binom{n}{k},$ which implies that $\bar{a_0}(rw)^{n}+\bar{a}_{n-k}/\binom{n}{k}(rw)^{k}\neq 0.$ We take
$$\lambda=-\frac{a_n(rw)^{n}+\frac{a_k}{\binom{n}{k}}(rw)^{k}}{\bar{a}_0(rw)^{n}+\frac{\bar{a}_{n-k}}{\binom{n}{k}}(rw)^{k}}, $$
then $\lambda$ is a well defined complex number with $|\lambda|>1$ and with this choice of $\lambda,$ we obtain $M(rw)=0$ where $|w|\geq 1.$ This contradicts the fact that all the zeros of $M(rz)$ lie in $|z|<1.$ Thus \eqref{l1eq2} holds. Letting $r\rightarrow 1$ in \eqref{l1eq2} and using continuity, we obtain
 \begin{align}\label{l1eq2ff}
\left|a_nz^n+\frac{a_k}{\binom{n}{k}}z^k\right|\leq \left|\bar{a}_0z^n+\frac{\bar{a}_{n-k}}{\binom{n}{k}}z^k\right|=\left|\frac{a_{n-k}}{\binom{n}{k}}z^{n-k}+a_0\right|
\end{align}
for $|z|= 1.$ Again, since $|\bar{a}_0|\geq |\bar{a}_{n-k}|/\binom{n}{k}$ then the polynomial $\frac{a_{n-k}}{\binom{n}{k}}z^{n-k}+a_0$ does not vanish in $|z|<1.$ Hence by the maximum modulus principle, 
 \begin{align*}
 \left|a_nz^n+\frac{a_k}{\binom{n}{k}}z^k\right|<\left|\frac{a_{n-k}}{\binom{n}{k}}z^{n-k}+a_0\right|\quad \text{for}\quad |z|<1.
 \end{align*}
 A direct application of Rouch\'{e}'s theorem shows that with $P(z)=a_{n}z^{n}+\cdots+a_{0},$
 \begin{align*}
 C_{\boldsymbol\gamma}P(z)=\left(a_nz^n+\frac{a_k}{\binom{n}{k}}z^k\right)e^{i\phi}+\left(\frac{a_{n-k}}{\binom{n}{k}}z^{n-k}+a_0\right)
 \end{align*}
 has all its zeros in $|z|\geq 1.$ Therefore, $C_{\boldsymbol\gamma}$ is an admissible operator. Applying \eqref{l3e} of Lemma \ref{l3}, the desired result follows immediately for each $p > 0$.
\end{proof}
The next lemma can be found in \cite{sg}.
\begin{lemma}\label{sl1}
Let $\alpha,\beta\in\mathbb{C}$ then for each $p>0,$
\begin{align}\label{sl1e}
\left\|\alpha z+\beta \right\|_p\geq \frac{|\alpha|+|\beta|}{2}\left\|1+z\right\|_p.
\end{align}
\end{lemma}

\section{Proof of Theorem \ref{t3}}

\begin{proof}[Proof of Theorem \ref{t3}]
Since $P(z)=\sum_{j=0}^{n}a_jz^j$ does not vanish in $|z|<\rho,$ $\rho\geq 1$ by Lemma \ref{nl1} for $s=0,1,2,\ldots, n-1,$ we have
\begin{align}\label{t2p1}
\rho^s\left|a_ne^{in\theta}+\frac{a_s}{\binom{n}{s}}z^{is\theta}\right|\leq \left|\frac{a_{n-s}}{\binom{n}{s}}e^{i(n-s)\theta}+a_0\right|.
\end{align}
Also, by Lemma \ref{l2},
\begin{align}\label{t2p2}
\int_{0}^{2\pi}\left|G(\theta)+e^{i\phi}Q(\theta)\right|^p\diff\theta\leq \Lambda^p\int_{0}^{2\pi}|P(e^{i\theta})|^p\diff\theta,
\end{align}
where 
$G(\theta)=a_ne^{in\theta}+\dfrac{a_s}{\binom{n}{s}}e^{is\theta}$ and $Q(\theta)=\dfrac{a_{n-s}}{\binom{n}{s}}e^{i(n-s)\theta}+a_0.$

Integrating both sides of \eqref{t2p2} with respect to $\phi$ from $0$ to $2\pi,$ we get for each $p>0$ 
\begin{align}\label{t2p3}
\int_{0}^{2\pi}\int_{0}^{2\pi}\left|G(\theta)+e^{i\phi}Q(\theta)\right|^p\diff\phi \diff\theta\leq \int_{0}^{2\pi}\Lambda^p\diff\phi\int_{0}^{2\pi}|P(e^{i\theta})|^p\diff\theta.
\end{align}
Now for $t_{1}\geq t_{2}\geq 1$ and $p>0,$ we have
$$ \int_{0}^{2\pi}|t_{1}+e^{i\phi}|^p\diff\phi \geq \int_{0}^{2\pi}|t_{2}+e^{i\phi}|^p\diff\phi.$$
If $G(\theta)\neq 0,$ we take $t_{1}=|Q(\theta)|/|G(\theta)|$ and $t_{2}=\rho^s$ then by \eqref{t2p1} $t_{1}\geq t_{2}$ and we get
\begin{align*}
\int_{0}^{2\pi}\left|G(\theta)+e^{i\phi}Q(\theta)\right|^p\diff\phi=&|G(\theta)|^p\int_{0}^{2\pi}\left|1+e^{i\phi}\frac{Q(\theta)}{G(\theta)}\right|^p\diff\phi\\=&|G(\theta)|^p\int_{0}^{2\pi}\left|e^{i\phi}+\left|\frac{Q(\theta)}{G(\theta)}\right|\right|^p\diff\phi\\&\geq |G(\theta)|^p \int_{0}^{2\pi}|\rho^s+e^{i\phi}|^p\diff\phi.
\end{align*}
For $G(\theta)=0,$ this inequality is trivially true. Using this in \eqref{t2p3}, we conclude that
\begin{align*}
\int_{0}^{2\pi}|G(\theta)|^p\diff\theta \int_{0}^{2\pi}|\rho^s+e^{i\phi}|^p\diff\phi\leq\int_{0}^{2\pi}\Lambda^p\diff\phi\int_{0}^{2\pi}|P(e^{i\theta})|^p\diff\theta.
\end{align*}
This implies
\begin{align*}
&\left\{\frac{1}{2\pi}\int_{0}^{2\pi}\left|a_ne^{in\theta}+\frac{a_s}{\binom{n}{s}}e^{is\theta}\right|^{p}\diff\theta\right\}^{1/p} \\ &\qquad\qquad\qquad\qquad\leq\frac{\left\{\frac{1}{2\pi}\int_{0}^{2\pi}\Lambda^p\diff\phi\right\}^{1/p}\left\{\frac{1}{2\pi}\int_{0}^{2\pi}\left|P(e^{i\theta})\right|^{p}\diff\theta\right\}^{1/p}}{\left\{\frac{1}{2\pi}\int_{0}^{2\pi}|\rho^s+e^{i\phi}|^p\diff\phi\right\}^{1/p}}. 
\end{align*}
which in conjunction with the fact that $\|az^{n}+b\|_p=\|az+b\|_p$ gives
\begin{align*}
\left\|a_nz+\frac{a_s}{\binom{n}{s}}\right\|_p\leq c_p\left\|1+z\right\|_p \frac{\left\|P\right\|_p}{\left\|z+\rho^s\right\|_p}
\end{align*}
where $ c_p$ is given by \eqref{cp}. This proves Theorem \ref{t3} for $p>0.$ To obtain this result for $p=0,$ we simply make $p\rightarrow 0+.$
\end{proof}
 
\end{document}